\documentclass[10pt]{amsart}
\usepackage{eurosym}
\usepackage{amsfonts}
\usepackage{amssymb}
\usepackage{graphicx}
\usepackage{amsmath,amsfonts,amsmath,amssymb}
\usepackage{mathrsfs}
\usepackage[latin1]{inputenc}
\usepackage[T1]{fontenc}
\usepackage{color,xspace}
\usepackage{enumerate}
\usepackage{comment}
\usepackage[usenames,dvipsnames]{pstricks}
\newtheorem{theorem}{Theorem}[section]

\newtheorem{lem}[theorem]{Lemma}

\newtheorem{proposition}[theorem]{Proposition}

\newtheorem{Assumption}[theorem]{Assumption}
\newcommand{\R}{\mathbb{R}}
\newcommand{\C}{\mathbb{C}}

\newcommand{\N}{\mathbb{N}}

\newcommand{\cA}{\mathcal A}
\newcommand{\cP}{\mathcal P}

\newcommand{\pf}{\mathfrak{p}}
\newcommand{\qf}{\mathfrak{q}}
\newcommand{\af}{\mathfrak{a}}
\newcommand{\norm}[1]{\left\Vert#1\right\Vert}

\newcommand{\abs}[1]{\left\vert#1\right\vert}
\usepackage{times}

\usepackage{epsfig}

\title{Determining the potential and the gradient coupling of two-state quantum systems in an infinite
waveguide}
\begin{document}

\maketitle

\begin{center}
Mohamed Hamrouni, Imèn Rassas, Éric Soccorsi
\end{center}
\bigskip
\bigskip
\bigskip
\bigskip

\textbf{Abstract.} 
We consider the inverse coefficient problem of simultaneously determining the space dependent electric potential, the zero-th order coupling term and the first order coupling vector of a two-state Schr\"odinger equation in an infinite cylindrical domain of $\R^n$, $n \ge 2$, from finitely many partial boundary measurements of the solution. We prove that these $n+1$ unknown scalar coefficients can be H\"older stably retrieved by $(n+1)$-times suitably changing the initial condition attached at the system.

\medskip


\medskip


%%%%%%%%%%%%%%%%%%%%%%%%%%%%%%%%%%%%%%%%%%%%%%%%%%%%%%%%%%%%%%%%%%
%%%%%%%%%%%%%%%%%%%%%%%%%%%%%%%%%%%%%%%%%%%%%%%%%%%%%%%%%%%%%%%%%%
%%%%%%%%%%%%%%%%%%%%%%%%%%%%%%%%%%%%%%%%%%%%%%%%%%%%%%%%%%%%%%%%%%
\section{Introduction}
This is the second of two papers dealing with the stability issue in the inverse problem of determining the electric potential and the coupling coefficients of a two-state quantum system, from local Neumann data. Such systems are commonly used to describe the dynamics of elementary particules like electrons, carrying a two-state (or two-level) quantum mechanical label called spin, and they can exist in any superposition of two independent states. Their dynamics are governed by Schr\"odinger equations bound together through a linear gradient coupling. When the two quantum states are constrained to a bounded spatial domain, it was proved in \cite{KRS} that the electric potential and the coupling are stably determined by finitely many partial boundary observations of the system. In the present work we aim for the same type of identification result when the quantum motion is no longer bounded but may escape to infinity in one direction over the course of time.

\subsection{Settings}

Throughout this article,
$\omega$ is a bounded domain of $\mathbb{R}^{n-1}$, $n \ge 2$, with smooth boundary $\gamma:=\partial\omega$, and $\Omega:= \omega \times \R$.
For $T\in \R_+$, we consider the following initial-boundary value problem (IBVP) with initial states $u_0^\pm$ and non-homogenous Dirichlet boundary conditions $g^\pm$, for the coupled Schr\"odinger equations in the unknowns $u^\pm$,
\begin{equation}
\label{sy1}
 \left\{
\begin{array}{ll}
-i\partial_t u^+ -\Delta u^++q^+u^+ +A\cdot\nabla u^-+pu^-=0&\textrm{in}\ Q:=\Omega \times (0,T) \\
-i\partial_t u^- -\Delta u^-+q^-u^- -A\cdot\nabla u^++pu^+=0&\textrm{in}\ Q \\
u^+(\cdot,0)=u^+_0,\ u^-(\cdot,0)=u^-_0&\textrm{in}\ \Omega\\
u^+=g^+,\ u^-=g^-&\textrm{on}\ \Sigma:= \Gamma \times (0,T),
\end{array}
\right.
\end{equation}
where $\Gamma := \gamma \times \R$. Since $\Gamma$ is unbounded, let us make the above boundary condition more precise. For all $x \in \Omega$, we write $x=(x^\prime,x_n)$ where $x^\prime =(x_1,\ldots,x_{n-1}) \in \omega$ and $x_n \in \R$, and using a standard density argument we extend the mapping
$$ \begin{array}{ccl} C_0^\infty(\R \times (0,T), H^2(\omega)) & \to & L^2(\R \times (0,T), H^{\frac{3}{2}}(\omega)) \\ w & \mapsto & \left[ (x_n,t) \in \R \times (0,T) \mapsto w(\cdot,x_n,t)_{| \gamma} \right], \end{array}
$$
to a bounded operator $\gamma_0$ acting from $L^2(\R \times (0,T),H^2(\omega))$ into 
$L^2(\R \times (0,T),H^{\frac{3}{2}}(\gamma))$. Then, for all $u^\pm \in L^2(0,T,H^2(\Omega))$, the boundary condition in \eqref{sy1} reads $\gamma_0 u^\pm = g^\pm$.

In the present paper we aim to stably retrieve the electric potentials $q^\pm : \Omega \to \R$, the zero-th order coupling term $p : \Omega \to \R$ and the first order coupling vector $A : \Omega \to \R^n$, by finitely many partial boundary measurements over 
the entire time-span $(0,T)$ of the solution $u^\pm$ to \eqref{sy1}. 
In contrast with \cite{KRS} where the spatial domain $\Omega$ is bounded, here we consider an infinitely extended cylindrical domain and we address the problem of simultaneous identification of non-compactly supported unknown coefficients $p$, $q^\pm$ and $A$. This requires a slightly different and technically more demanding approach than the one implemented in \cite{KRS}.  

\subsection{Motivations}
The dynamics of the two states $u^\pm$ governed by \eqref{sy1} are bound together through linear gradient coupling $p u^\mp  \pm A \cdot \nabla u^\mp$. We refer the reader to \cite{LSY} and the references therein for the relevance of these processes in physics. Gradient coupling 
appears also naturally in quantum fields theory (see \cite{BLMM,R}) or quantum cosmology (see \cite{BTW,DKT}), and it is sometimes a first-order approximation of nonlinear coupling (see \cite{Y}).

Quantum wires such as carbon nanotubes are extremely narrow structures which have a length-to-diameter ratio up to $10^8$. They are commonly modeled by infinite three-dimensional cylindrical domains such as $\Omega$, in which electrons are essentially free to move in one direction. Quantum wires exhibit valuable physical properties for electronics, optics and other fields of materials science and technology, see e.g., \cite{AKBRSR}, and for this reason they have attracted a lot of attention from the scientific community.  

The IBVP \eqref{sy1} can be interpreted as the time-evolution of 
the spin of a spin-$\frac{1}{2}$ particle such as an electron (whose spin can have values $\pm\frac{\hbar}{2}$, where $\hbar$ is the reduced Planck constant) confined in a carbon nanotube. Notice that for the sake of notational simplicity, the various physical constants such as $\hbar$, the charge and the mass of the electron, are all taken equal to one in \eqref{sy1}.

\subsection{Bibliography}
The mathematical literature devoted to inverse coefficient problems for the dynamic Schr\"odinger equation is so extensive that this presentation is not intended to be exhaustive, but we can mention \cite{B,BC, BKS2, CKS, KS} where zero-th or/and first order unknown coefficients of the Schr\"odinger equation are determined by the Dirichlet-to-Neumann map. These articles assume knowledge of infinitely many boundary data, but in \cite{BP,YY} the real-valued electric potential is stably retrieved by one partial lateral observation of the solution. This result was 
%improved to smaller partial boundary measurements in \cite{MOR} and 
extended to complex-valued electric potentials in \cite{HKSY}. The boundary measurement in \cite{BP,HKSY,YY} is taken on a subpart of the boundary fulfilling a geometric condition related to geometric optics condition insuring observability. This condition was relaxed to arbitrarily small sub-boundaries in \cite{BC}, provided the potential is known in the vicinity of the boundary. 
The inverse problem of determining the magnetic vector potential of the autonomous Schr\"odinger equation is addressed in \cite{HKSY}. The same problem for the space-varying part of the magnetic potential appearing in a non-autonomous Schr\"odinger equation is treated in \cite{CS}. In both cases, the $n$-th dimensional unknown magnetic vector potential, $n \geq 1$, is retrieved from $n$ partial Neumann data obtained by $n$-times suitably selecting the initial condition attached at the magnetic Schr\"odinger equation.

The strategy of \cite{BP,CS,HKSY,YY} relies on a Carleman inequality specifically designed for the Schr\"odinger equation, see \cite{HKSY,T,YY} for actual examples of such weighted energy estimates. 
The idea of using a Carleman estimate for solving inverse problems goes back to 1981 and was introduced by A. L. Bukhgeim and M. V. Klibanov in their seminal article \cite{BK}. Since then, the Bukhgeim-Klibanov approach has been successfully applied to parabolic, hyperbolic and Schr\"odinger systems and even to coupled systems of partial differential equations. We refer the reader to \cite{K} and references therein, for a complete survey of multidimensional inverse problems solved by the Bukhgeim-Klibanov method. 

In all the aforementioned papers, the Schr\"odinger equation under study is posed in a bounded spatial domain. The inverse problem of determining the electric potential of the Schr\"odinger equation stated in an infinite waveguide is examined in \cite{BKS1,KPS2}. This is achieved by mean of a specifically designed Carleman estimate for the Schr\"odinger equation in an unbounded cylindrical domain, which is established in \cite{KPS1}. 
All the articles listed above are concerned with the "one state" Schr\"odinger equation. In \cite{LT}, assuming that the gradient coupling vector is known, the authors show that the zero-th order coupling term of a two state magnetic Schr\"odinger equation is uniquely determined by one partial Neumann data. Recently in \cite{ZD}, the electric potential of a strongly coupled Schr\"odinger equations in a bounded spatial domain was Lipschitz stably retrieved by one partial (internal or boundary)
measurement of the solution to the system. In \cite{KRS}, the zero-th and first order coefficients of the coupling are H\"older stably recovered by finitely many partial boundary observations of the solution. The coupled Schr\"odinger equations under study in
in \cite{KRS,LT,ZD} are posed on a bounded spatial domain. In the present paper, we aim to extend the result of \cite{KRS} to the case of an unbounded waveguide.

\subsection{Notations}
Throughout this text $x = (x_1,...,x_n)$ is a generic point of $\overline{\Omega}$ that is sometimes written $x=(x^\prime,x_n)$ where $x^\prime=(x_1,\cdots,x_{n-1}) \in \overline{\omega}$ is the variable of the transverse section of $\Omega$ and $x_n \in \R$ is the longitudinal variable. For all $x=(x^\prime,x_n)\in\Gamma$, the outward unit normal $\nu$ to $\Gamma$ reads
$\nu(x)=\nu(x^\prime)=(\nu^\prime(x^\prime),0)^T$, 
where $\nu^\prime(x^\prime)\in\R^{n-1}$ is the outgoing normal vector to $\gamma$ at $x^\prime$ and $a^T$ denotes the transpose of the row vector $a$.

For all $i=1,\ldots,n$ we set $\partial_i := \frac{\partial}{\partial x_i}$ in such a way that $\nabla:=(\partial_1,\ldots,\partial_n)^T$ (resp., $\nabla^\prime:=(\partial_1,\ldots,\partial_{n-1})^T$) is the gradient operator with respect to $x=(x_1,\ldots,x_n)$ (resp., $x^\prime=(x_1,\ldots,x_{n-1})$). Similarly, we write $\partial_t = \frac{\partial}{\partial t}$. 
For the sake of shortness we write 
$\partial_{ij}^2$, $i,j=1,\ldots,n$, instead of $\partial_i \partial_j$ and as usual we denote by $\Delta$ the Laplace operator $\partial_1^2+\ldots+\partial_n^2$. Next, for any multi-index $k = (k_1,\ldots,k_n) \in \N_0^n$, where $\N_0 := \{ 0 \} \cup \N$, we put $\abs{k}:=k_1+\ldots+k_n$ and $\partial_x^k= \partial_1^{k_1} \ldots \partial_n^{k_n}$. .

Further, the symbol $\cdot$ denotes the scalar product in $\C^m$, $m \in \N$, and we set $\abs{\zeta}:= \sqrt{\zeta \cdot \zeta}$ for all $\zeta \in \C^m$. We simply write $\nabla \cdot$ for the divergence operator in $\R^n$ and we set $\partial_\nu u := \nabla u \cdot \nu = \nabla^\prime \cdot \nu^\prime$. 

Finally, for all $r>0$ and $s>0$, we introduce
$H^{r,s}(\Sigma) := L^2(0,T;H^r(\Gamma)) \cap H^s(0,T;L^2(\Gamma))$
where $H^s(\Gamma)$ denotes the usual Sobolev space on $\Gamma$ of order $s$. 

\subsection{Main results}
Prior to investigating the inverse problem under study in this article, we examine the well-posedness issue for the forward problem associated with \eqref{sy1}. For this purpose we introduce the Hamiltonian operator acting on $(C_0^\infty(Q)^\prime)^2$,
$$ \mathcal{H}(A,p,q^\pm):= \left( \begin{array}{cc} -\Delta  + q^+ & A \cdot \nabla + p \\ -A \cdot \nabla +p & -\Delta+ q^- \end{array} \right) $$
and state the following existence, uniqueness and regularity result for the solution to the IBVP
\eqref{sy1}.
%Arguing as in the derivation of \cite[Lemma 2.3]{1}, we obtain the following existence and uniqueness result.

\begin{proposition}
\label{pr1}
Let $m \in \N$ and assume that $\gamma$ is $\mathcal{C}^{2(m+1)}$. Let $A \in W^{2m+1,\infty}(\Omega,\R^n) \cap C^{2(m-1)}(\overline{\Omega},\R^n)$ be such that $\nabla \cdot A=0$ a.e. in $\Omega$, let
$p \in W^{2m+1,\infty}(\Omega,\R) \cap C^{2(m-1)}(\overline{\Omega},\R)$ and let $q^\pm \in W^{2m+1,\infty}(\Omega,\R) \cap C^{2(m-1)}(\overline{\Omega},\R)$ satisfy
$$ \norm{A}_{W^{2m+1,\infty}(\Omega)}+ \norm{p}_{W^{2m+1,\infty}(\Omega)}+\norm{q^+}_{ W^{2m+1,\infty}(\Omega)}+ \norm{q^-}_{W^{2m+1,\infty}(\Omega)} \leq M,$$
for some {\it a priori} fixed positive constant $M$. 
Then, for all
$g=(g^+,g^-)^T \in H^{2(m+7/4),m+7/4}(\Sigma)^2$ and all $u_0=(u_0^+,u_0^-)^T \in H^{2m+3}(\Omega)^2$ fulfilling the following compatibility conditions
\begin{equation}
\label{d4}
\partial_t^\ell g(\cdot,0)=( -i )^\ell \mathcal{H}(A,p,q^\pm)^\ell u_0\ \mbox{on}\ \Gamma,\ \ell=0,\cdots,m,
\end{equation}
the IBVP \eqref{sy1} admits a unique solution $u=(u^+,u^-)^T \in \cap_{\ell=0}^{m+1} H^{m+1-\ell}(0,T;H^{2 \ell}(\Omega)^2)$. Moreover,  there exists a positive constant $C$, depending only on $\omega$, $T$ and $M$ %, $u_0$ and $g$, 
such that
\begin{equation}
\label{d4c}
\sum_{\ell=0}^{m+1} \norm{u}_{H^{m+1-\ell}(0,T;H^{2 \ell}(\Omega)^2)} \leq C \left( \norm{u_0}_{H^{2m+3}(\Omega)^2}+ \norm{g}_{H^{2(m+7/4),m+7/4}(\Sigma)^2}\right).
\end{equation}
\end{proposition}

Notice that the divergence-free condition on $A$ requested by Proposition \ref{pr1} is to guarantee that
$\mathcal{H}(A,p,q^\pm)$ endowed with homogeneous Dirichlet boundary condition on $\Gamma$, has a self-adjoint realization $H(A,p,q^\pm)$ in $L^2(\Omega)^2$, see \cite[Lemma 2.1]{KRS}. As a consequence the operator $-i H(A,p,q^\pm)$ is $m$-dissipative in $L^2(\Omega)^2$, and since the IBVP \eqref{sy1} is equivalently rewritten as 
$$ \left\{ \begin{array}{ll} -i \partial_t u +\mathcal{H}(A,p,q^\pm) u = 0 & \mbox{in}\ Q \\ u(\cdot,0)=u_0 & \mbox{in}\ \Omega \\ u=g & \mbox{on}\ \Sigma, \end{array} \right. $$
the statement of Proposition \ref{pr1} follows by arguing in the same way as in the proof of \cite[Lemma 2.3]{KRS}. 

We point out that the regularity assumptions on the coefficients $A$, $p$ and $q^\pm$, the initial states $u_0^\pm$ and the boundary conditions $g^\pm$, in Proposition \ref{pr1}, are only sufficient conditions ensuring a higher order of regularity of the solution $u^\pm$ to \eqref{sy1}, as requested by the analysis of the inverse problem under study in this article. As a matter of fact the Bukhgeim-Klibanov method requires $\partial_t u^\pm$ and $\partial_t \nabla u^\pm$ to be bounded in $Q$, which can be achieved upon taking $m$ in Proposition \ref{pr1}, sufficiently large relative to $n$. Namely, we choose
\begin{equation}
\label{def-N}
N \in \N \cap \left( \frac{n+2}{4}+1 ,\frac{n+2}{4}+ 2 \right],
\end{equation}
pick $M$, $\kappa$, $\varrho$, $\af$, $\pf$ and $\qf$ in $\R_+$, and for $A_0\in W^{2N+1,\infty}(\Omega,\R^n)\cap C^{2(N-1)}(\overline{\Omega},\R^n)$, $p_0 \in W^{2N+1,\infty}(\Omega,\R)\cap C^{2(N-1)}(\overline{\Omega},\R)$ and $q_0^\pm \in W^{2N+1,\infty}(\Omega,\R)\cap C^{2(N-1)}(\overline{\Omega},\R)$, we introduce the set of unknown electric potentials
as
\begin{eqnarray}
\label{pco}
\cP_{\pf}(p_0) & := & \left\{p\in W^{2N+1,\infty}(\Omega,\R) \cap C^{2(N-1)}(\overline{\Omega},\R)\ \mbox{s.t.}\ \norm{p}_{W^{2N+1,\infty}(\Omega)} \leq M, \right.\\
& & \left. \partial_x^k p = \partial_x^k p_0\ \mbox{on}\ \Gamma,\ k=0,\ldots, 2(N-1)\ \mbox{and}\ \abs{(p-p_0)(\cdot,x_n)} \leq \pf e^{-\kappa \langle x_n\rangle^{\varrho}},\ x_n \in \R \right\},
\nonumber
\end{eqnarray}
the set of unknown zero-th order coupling coefficients as $\cP_{\qf}(q_0^\pm)$, 
%\begin{eqnarray}
%\label{qco}
%\cP^\pm(q^\pm_0) & := & \left\{ q^\pm\in W^{N+1,\infty}(\Omega,\R)\ \mbox{s.t.}\ \norm{q^\pm}_{W^{N+1,\infty}(\Omega)}\leq M \right. \\
%& & \left.  \mbox{and}\ \abs{(q^\pm-q_0^\pm)(\cdot,x_n)} \leq \qf e^{-\kappa \langle x_n\rangle^{\varrho}},\ x_n \in \R \right\},
%\end{eqnarray}
and the set of unknown first order coupling vectors as
\begin{eqnarray}
\label{aco}
\cA_\af(A_0) & := & \left\{A \in W^{2N+1,\infty}(\Omega,\R^n)\cap C^{2(N-1)}(\overline{\Omega},\R)\ \mbox{s.t.}\ \norm{A}_{W^{2N+1,\infty}(\Omega)^n}\leq M,\ \nabla \cdot A=0\ \mbox{in}\ \Omega, \right. \\
& &  \left. \ \partial_x^k A = \partial_x^k A_0\ \mbox{on}\ \Gamma,\ \abs{k}=0,\ldots, 2(N-1)\  \mbox{and}\  \abs{(A-A_0)(\cdot,x_n)} \leq \af e^{-\kappa \langle x_n\rangle^{\varrho}},\ x_n \in \R \right\}. \nonumber
\end{eqnarray}
Here, the notation $\partial_x^k$ for $\abs{k}=m \in \N_0$ is a shorthand for $\partial_1^{k_1} \ldots \partial_n^{k_n}$ where $k=(k_1,\ldots,k_n) \in \N_0^n$ satisfies $\abs{k}=k_1+\ldots+k_n=m$. 

Then, the main result of this article can be stated as follows. 

\begin{theorem}
\label{thm-main}
Assume that $\gamma$ is $\mathcal{C}^{2(N+1)}$. For $j=1,2$, let $A_j\in\cA_\af(A_0)$ satisfy \begin{equation}
\label{h1}
\exists y_* \in \R_+,\ a_{1,n}(x^\prime,x_n)=a_{2,n}(x^\prime,x_n),\ x^\prime \in \omega,\ x_n \in (-y_*,y_*),
\end{equation}
let $p_j\in\cP_\pf(p_0)$ and let $q_j^\pm \in\cP_\qf(q^\pm_0)$. 

Then, there exist a sub-boundary $\gamma_* \subset \partial \omega$ and a set of $n+1$ initial states $u_0^k=(u_0^{+,k},u_0^{-,k})^T \in H^{2N+3}(\Omega)^2$ and boundary conditions $g^k=(g^{+,k},g^{-,k})^T \in H^{2(N+7 \slash 4), N + 7 \slash 4}(\Sigma)^2$, $k=1,\ldots, n+1$, fulfilling the compatibility conditions
\begin{equation}
\label{cck}
%\partial_t^\ell g^k(\cdot,0)=( -i )^\ell \left( \begin{array}{cc} -\Delta  + q_0^+ & A_0 \cdot \nabla + p_0 \\ -A_0 \cdot \nabla +p_0 & -%\Delta+ q_0^- \end{array} \right)^\ell u_0^k\ \mbox{on}\ \Gamma,\ \ell=0,\cdots,N,
\partial_t^\ell g^k(\cdot,0)=( -i )^\ell \mathcal{H}(A_0,p_0,q_0^\pm)^\ell u_0^k\ \mbox{on}\ \Gamma,\ \ell=0,\cdots,N,
\end{equation}
such that for all $\theta \in \left( 0 , \frac{1}{2} \right)$, the following estimate
\begin{eqnarray}
\label{se}
& & \norm{A_1-A_2}^2_{L^2(\Omega)}+\norm{p_1-p_2}^2_{L^2(\Omega)}+\norm{q_1^+-q_2^+}^2_{L^2(\Omega)}+\norm{q_1^- - q_2^-}^2_{L^2(\Omega)}\\
& \leq & C \sum_{k=1}^{n+1} \left( \norm{\partial_\nu \partial_t u_1^{-,k}- \partial_\nu \partial_t u_2^{-,k}}^{\theta}_{L^2(\Sigma_*)}
+\norm{\partial_\nu \partial_t u_1^{+,k} -\partial_\nu \partial_t u_2^{+,k}}^\theta_{L^2(\Sigma_*)} \right), \nonumber
\end{eqnarray}
holds for some positive constant $C$ depending only on $\omega$, $T$, $\gamma^*$, $M$, $y_*$, $\theta$, $\kappa$, $\varrho$, $\af$, $\pf$, $\qf$ and $(u_0^{\pm,k},g^{\pm,k})$, $k=1,\ldots,n+1$. Here, $\Sigma_*:=\gamma_*\times\R\times(0,T)$ and $u_j^k=(u_j^{+,k},u_j^{-,k})^T$, for $j=1,2$, is the solution to \eqref{sy1} given by Proposition \ref{pr1}, where $(A_j,p_j,q_j^\pm,u_0^{\pm,k},g^{\pm,k})$ is substituted for $(A,p,q^\pm,u_0^{\pm},g^{\pm})$.
\end{theorem}

\subsection{Brief comments}
Theorem \ref{thm-main} claims that $n+2$ Neumann data stably determine $n+2$ unknown scalar coefficients (strictly speaking there are $n+3$ unknown scalar coefficients in the inverse problem that Theorem \ref{thm-main} is dealing with, but since the $n$ components of the gradient coupling vector are bound together through some divergence free condition, they only amount for $n-1$ free unknown scalar coefficients). From the viewpoint of the analysis of inverse problems, such a result is thus somehow optimal.

The statement and the strategy of the proof of Theorem \ref{thm-main} are very similar to the ones of \cite[Theorem 1.2]{KRS}, which holds for a bounded spatial domain $\Omega$. Nevertheless, there are two major differences in the derivation of Theorem \ref{thm-main} as compared to the one of \cite[Theorem 1.2]{KRS}. Firstly, the Carleman estimate that is used in Section \ref{sec-pr} below is designed for a Schr\"odinger equation in an unbounded cylindrical domain, and it is slightly different from the one used in \cite{KRS}, which is specific to the Schr\"odinger equation in a bounded domain. Secondly, the construction of the initial states $u_0$ used for probing the system in the analysis of the inverse problem under examination in this article, is more delicate than in \cite{KRS}. This is due to the fact that it is technically more challenging to design a suitable set of $L^2(\Omega)$-initial states $u_0$ when the domain $\Omega$ is infinitely extended, than in the case where $\Omega$ is bounded. As can be seen from Section \ref{sec-pr}, this $L^2$-condition will be satisfied by introducing an additional decay with respect to the infinite direction of the waveguide.

\subsection{Outline}
The paper is designed as follows: In the following section we collect several technical results needed for the proof of Theorem \ref{thm-main}, which is given in Section \ref{sec-pr}.

\section{Preliminaries}
We first establish that the solution to \eqref{sy1} is bounded in $Q$.

\subsection{Boundedness of the solution}
The result we have in mind is as follows.

\begin{lem}
\label{lm1}
Assume that conditions of Proposition \ref{pr1} are satisfied with $m=N$, where $N$ is the same as in \eqref{def-N}. Then, the solution $u$ to \eqref{sy1} lies in $W^{1,\infty}(0,T;W^{1,\infty}(\Omega)^2)$ and satisfies
$$
\norm{u}_{W^{1,\infty}(0,T;W^{1,\infty}(\Omega)^2)} \leq C,
$$
for some positive constant $C$ depending only on $\omega$, $T$, $M$, $u_0$ and $g$.
\end{lem}
\begin{proof}
We have $u \in H^2(0,T,H^{2(N-1)}(\Omega)^2)$ by Proposition \ref{pr1}, with $2(N-1)>\frac{n}{2}+1$ from \eqref{def-N}.
Since $H^k(\Omega)$ is continuously embedded in $L^\infty(\Omega)$ for all $k> \frac{n}{2}$, according to \cite[Lemma 2.7]{KPS2} (which extends the corresponding well-known Sobolev embedding theorem in $\R^n$, see e.g. \cite[Corollary IX.13]{Bre} or \cite[Section 5.10, Problem 18]{Eva}, to the case of the unbounded cylindrical domain $\Omega$), the result follows from this and \eqref{d4c}.
\end{proof}

\subsection{Global Carleman estimate for the Schr\"odinger equation in $\omega \times \R$}
For further use we introduce a global Carleman estimate specifically designed for the Schr\"odinger equation in the unbounded cylindrical domain $\Omega$, which is borrowed from \cite[Proposition 3.3 and Lemma 4.2]{KPS1}. 
%As will appear in Section {sec-pr} below, this weighted energy inequality is the main tool in the derivation of the main result of this article.

For this purpose we pick a function $\alpha\in C^4(\overline{\omega},\R_+)$ and an open subset $\gamma_*\subset\partial\omega$ satisfying the following conditions:
\begin{Assumption}
\hfill %\break
\label{hyp}
\begin{enumerate}[(i)]
\item $\exists c \in \R_+$ s.t. $\abs{\nabla^\prime\alpha(x^\prime)}\geq c$ for all $x^\prime\in\omega$. 
\item $\forall x^\prime\in \gamma \setminus\gamma_*$, $\partial_\nu\alpha(x^\prime)=\nabla^\prime\alpha(x^\prime)\cdot\nu^\prime(x^\prime)<0$.
\item $\exists \lambda_0 \in \R_+$,\ $\exists c \in \R_+$ s.t. 
$$\lambda \abs{\nabla^\prime\alpha(x^\prime)\cdot\zeta}^2+D^2\alpha(x^\prime,\zeta)\geq c \abs{\zeta}^2,\ \zeta\in\R^{n-1},\ x^\prime\in\omega,\ \lambda \ge \lambda_0, $$ 
where $D^2\alpha(x^\prime):=\left( \partial_{i,j}^2 \alpha(x^\prime)\right)_{1\leq i,j\leq n-1}$ and $D^2\alpha(x^\prime,\zeta)$ denotes the $\R^{n-1}$-scalar product of $D^2\alpha(x^\prime)\zeta$ with $\zeta$.
\end{enumerate}
\end{Assumption}
We point out that there exist $\alpha$ and $\gamma_*$ fulfilling the above conditions (i), (ii) and (iii). As a matter of fact, for all $x^\prime_0\in\R^{n-1}\setminus\overline{\omega}$ fixed, this is the case of the function $\alpha(x^\prime)=\abs{x^\prime-x^\prime_0}^2$ and any open subset $\gamma_*\subset \gamma$ such that $\{x^\prime\in \gamma;\ (x^\prime-x^\prime_0)\cdot\nu(x^\prime)\geq 0\}\subset\gamma_*$. 

Next, putting $K:=r \norm{\alpha}_{L^\infty(\omega)}$ for some $r \in (1,+\infty)$, we set
\begin{equation}
\label{ee1}
\beta(x):=\alpha(x^\prime)+ K,\ x=(x^\prime,x_n)\in\Omega,
\end{equation}
and 
%for all $\lambda \in \R_+$ 
we introduce the following weight functions on $\tilde{Q}:=\Omega \times (-T,T)$:
\begin{equation}
\label{k2}
%\varphi(x,t):=\frac{e^{2\lambda\beta(x)}}{(T+t)(T-t)}\ \mbox{and}\ \eta(x,t):=\frac{e^{2\lambda K}-e^{\lambda \beta(x)}}{(T+t)(T-t)},\ (x,t)
\varphi(x,t):=\frac{e^{2\beta(x)}}{(T+t)(T-t)}\ \mbox{and}\ \eta(x,t):=\frac{e^{2 K}-e^{ \beta(x)}}{(T+t)(T-t)},\ (x,t)\in \tilde{Q}.
\end{equation}
Let us notice for further use that
\begin{equation}
\label{a1}
\eta(x,t) \geq \eta_0(x) >0,\ (x,t) \in \tilde{Q},
\end{equation}
where $\eta_0(x):= \eta(0,x)$ for all $x \in \Omega$.
%Having defined the functions $\beta$, $\varphi$ and $\eta$ in \eqref{ee1}-\eqref{k2}, 
This being said, we may now state the global Carleman estimate established in \cite[Proposition 3.3 and Lemma 4.2]{KPS1}.

\begin{proposition}\label{pro1}
Suppose that $\alpha$ and $\gamma_*$ fulfill Assumption \ref{hyp}. Let $\beta$ be as in \eqref{ee1} and let $\varphi$ and $\eta$ be defined by \eqref{k2}. Then, there exist two constants $s_0>0$ and $C>0$, depending only on $T$, $\omega$ and $\gamma_*$, such that the estimate
\begin{eqnarray*}
& &s^{-1/2}\norm{e^{-s\eta}\nabla^\prime w}^2_{L^2(\tilde{Q})}+s^{-1/2}\norm{e^{-s\eta}w}^2_{L^2(\tilde{Q})}+\norm{e^{-s\eta_0}w(\cdot,0)}^2_{L^2(\Omega)} \\
&\leq & Cs^{-3/2}\left(s\norm{e^{-s\eta}\varphi^{1/2} \abs{\partial_\nu\beta}^{\frac{1}{2}}\partial_\nu w}^2_{L^2(\tilde{\Sigma}_*)}+\norm{e^{-s\eta}L w}^2_{L^2(\tilde{Q})}\right)
\end{eqnarray*} 
holds whenever $s\geq s_0$ and $w\in L^2(-T,T;H^1_0(\Omega))$ satisfies $L w\in L^2(\tilde{Q})$ and $\partial_\nu w\in L^2(\tilde{\Sigma}_*)$. Here, $\tilde{\Sigma}_*:=(-T,T)\times\Gamma_*$ and $\Gamma_*:=\gamma_*\times\R$.
\end{proposition}

Armed with Proposition \ref{pro1}, we turn now to proving the main result of this article.

\section{Proof of Theorem \ref{thm-main}}
\label{sec-pr}
%%%%%%%%%%%%%%%%%%%%%%%%%%%%%%%%%%%%%%%%%%%%%%%%%%%%%%%%%%%%%%%%%%%%%%%%%%%%%%%%%%%%%%%%%%%%
\subsection{Linearization, time-differentiation and all that}
We start by linearizing the system \eqref{sy1}. For this purpose we consider the two solutions $u_j=(u^+_j,u^-_j)^T$, $j=1,2$, to the IBVP \eqref{sy1} where $(A_j,p_j,q^\pm_j)$ is substituted for $(A,p,q^\pm)$. Then, $u^\pm:=u^\pm_1-u^\pm_2$ solves
\begin{equation}
\label{sy2}
 \left\{
\begin{array}{ll}
-i\partial_t u^+-\Delta u^++q^+_1u^+=-A_1\cdot\nabla u^--A\cdot\nabla u_2^--q^+u_2^+-p_1u^--pu_2^-&\textrm{in}\ Q \\
-i\partial_t u^- -\Delta u^-+q^-_1u^-=A_1\cdot\nabla u^++A\cdot\nabla u_2^+-q^-u_2^--p_1u^+-pu_2^+&\textrm{in}\ Q\\
u^+(\cdot,0)=0,\ u^-(\cdot,0)=0&\textrm{in}\ \Omega\\
u^+=0,\ u^-=0&\textrm{on}\ \Sigma,
\end{array}
\right.
\end{equation}
where $A:=A_1-A_2$, $p:=p_1-p_2$ and $q^\pm:=q^\pm_1-q^\pm_2$.
Further, $u^\pm$ lies in $H^2(0,T;L^2(\Omega)) \cap H^1(0,T;H^2(\Omega)\cap H^1_0(\Omega))$, we 
differentiate \eqref{sy2} with respect to the time-variable and find that
$$
\left\{
\begin{array}{ll}
-i\partial_tv^+-\Delta v^++q^+_1v^+=-A_1\cdot\nabla v^--A\cdot\nabla\partial_tu_2^--q^+\partial_tu_2^+-p_1 v^--p\partial_tu_2^-&\textrm{in}\ Q \\
-i\partial_tv^--\Delta v^-+q^-_1v^-=A_1\cdot\nabla v^++A\cdot\nabla\partial_tu_2^+-q^-\partial_tu_2^--p_1v^+-p\partial_tu_2^+&\textrm{in}\ Q \\
v^+(\cdot,0)=-i(A\cdot\nabla u_0^-+q^+u_0^++pu_0^-)&\textrm{in}\ \Omega\\
v^-(\cdot,0)=-i(-A\cdot\nabla u_0^++q^-u_0^-+pu_0^+)&\textrm{in}\ \Omega\\
v^+=0,\ v^-=0&\textrm{on}\ \Sigma,
\end{array}
\right.
$$
where $v^\pm:=\partial_t u^\pm$.
The next step is to extend $u^\pm_2$ to $\tilde{Q}=\Omega\times(-T,T)$ by setting $u^\pm_2(x,t):=\overline{u^\pm_2(x,-t)}$ for a.e. $(x,t)\in \Omega\times(-T,0)$. Since $u_0^{\pm}$, $A$, $p$ and $q^\pm$ are-real valued, it is not hard to see that the function $v^\pm$, extended to $\Omega\times(-T,0)$ as $v^\pm(x,t):=-\overline{v^\pm(x,-t)}$, satisfies
\begin{equation}
\label{sy3}
\left\{
\begin{array}{ll}
-i\partial_tv^+-\Delta v^++q^+_1v^+=-A_1\cdot\nabla v^--A\cdot\nabla\partial_tu_2^--q^+\partial_tu_2^+-p_1 v^--p\partial_tu_2^-&\textrm{in}\ \tilde{Q} \\
-i\partial_tv^--\Delta v^-+q^-_1v^-=A_1\cdot\nabla v^++A\cdot\nabla\partial_tu_2^+-q^-\partial_tu_2^--p_1v^+-p\partial_tu_2^+&\textrm{in}\ \tilde{Q} \\
v^+(\cdot,0)=-i(A\cdot\nabla u_0^-+q^+u_0^++pu_0^-)&\textrm{in}\ \Omega\\
v^-(\cdot,0)=-i(-A\cdot\nabla u_0^++q^-u_0^-+pu_0^+)&\textrm{in}\ \Omega\\
v^+=0,\ v^-=0&\textrm{on}\ \tilde{\Sigma}:=\Gamma\times(-T,T).
\end{array}
\right.
\end{equation}
%%%%%%%%%%%%%%%%%%%%%%%%%%%%%%%%%%%%%%%%%%%%%%%%%%%%%%%%%%%%%%%%%%%%%%%%%%%%%%%%%%%%%%%%%%
Put $\mu^\pm:=\norm{e^{-s\eta_0} \varphi^{1/2}\abs{\partial_\nu\beta}^{1/2} \partial_\nu v^\pm}^2_{L^2(\tilde{\Sigma}_*)}$.
Then, applying Proposition \ref{pro1} to \eqref{sy3}, we get for all $s \ge s_0$ that
\begin{eqnarray}
& & s^{-1/2} \norm{ e^{-s\eta}\nabla^\prime v^\pm}^2_{L^2(\tilde{Q})} + s^{-1/2}\norm{e^{-s\eta}v^\pm}^2_{L^2(\tilde{Q})}+\norm{ e^{-s\eta_0}v^\pm(\cdot,0)}^2_{L^2(\Omega)}  \label{b1}\\
& \leq & C s^{-3/2}\left( s \mu^\pm+\norm{e^{-s\eta} \left( \pm A_1\cdot\nabla v^\mp \pm A\cdot\nabla\partial_tu^\mp_2+q^\pm \partial_tu^\pm_2+p_1 v^\mp+p\partial_tu^\mp_2 \right)}^2_{L^2(\tilde{Q})} \right),\nonumber
\end{eqnarray}
for some positive constant $C$ depending only on $\omega$, $T$ and $\gamma_*$. 
Taking into account that $\norm{A_1}_{L^\infty(\Omega)}\le M$ , $\norm{p_1}_{L^\infty(\Omega)} \le M$, and that the two functions $\partial_tu^\pm_2$ and $\nabla \partial_tu^\pm_2$ are bounded on $\tilde{Q}$ by some positive constant depending only on $\omega$, $T$, $M$, $u_0$ and $g$ according to Lemma \ref{lm1},  \eqref{a1} and \eqref{b1} then yield that
\begin{eqnarray*}
% \label{mp.}
& & s^{-1/2}\norm{e^{-s\eta}\nabla^\prime v^\pm}^2_{L^2(\tilde{Q})}+s^{-1/2}\norm{e^{-s\eta}v^\pm}^2_{L^2(\tilde{Q})}+\norm{e^{-s\eta_0}v^\pm(\cdot,0)}^2_{L^2(\Omega)} \\
%& \leq & C M^2 s^{-3/2}\left(s\mu^\pm+\Vert e^{-s\eta}\nabla_{x^\prime}v^\mp\Vert^2_{L^2(\tilde{Q})}+\Vert e^{-s\eta}A\Vert^2_{L^2(\tilde{Q})}+\Vert e^{-%s\eta}q^\pm\Vert^2_{L^2(\tilde{Q})}
%+\Vert e^{-s\eta}v^\mp\Vert^2_{L^2(\tilde{Q})}+\Vert e^{-s\eta}p\Vert^2_{L^2(\tilde{Q})}\right) \nonumber \\
& \leq &C s^{-3/2}\left( s\mu^\pm+\norm{e^{-s\eta}\nabla_{x^\prime}v^\mp}^2_{L^2(\tilde{Q})}+\norm{e^{-s\eta}v^\mp}^2_{L^2(\tilde{Q})}+\norm{e^{-s\eta_0}A}^2_{L^2(\Omega)^n}+\norm{e^{-s\eta_0}q^\pm}^2_{L^2(\Omega)}
+\norm{e^{-s\eta_0}p}^2_{L^2(\Omega)}\right), \nonumber
\end{eqnarray*}
provided $s \ge s_0$. Here and in the remaining part of this proof, $C$ denotes a generic positive constant which may change from line to line. Although the constant $C$ depends only on $\omega$, $T$, $\gamma_*$, $M$, $u_0$ and $g$ in the above estimate, in the sequel it might also depend on one or several of the parameters $n$, $y_*$, $\kappa$, $\varrho$, $\af$, $\pf$, $\qf$ and $\theta$ of the problem, as well. Nevertheless, we shall not systematically specify the dependence of $C$ with respect to the above mentioned parameters.

As a consequence we have
\begin{eqnarray*}
%\label{mp1.}
& & s^{-\frac{1}{2}} \left( 1 -C s^{-1} \right) \sum_{\ell=\pm} \left( \norm{e^{-s \eta}\nabla^\prime v^\ell}^2_{L^2(\tilde{Q})}+\norm{e^{-s \eta} v^\ell}^2_{L^2(\tilde{Q})} \right) +\sum_{\ell=\pm} \norm{e^{-s\eta_0}v^\ell(\cdot,0)}^2_{L^2(\Omega)} \\
&\leq & C s^{-\frac{3}{2}}\left( \norm{e^{-s\eta_0}A}^2_{L^2(\Omega)^n}+\norm{e^{-s\eta_0}p}^2_{L^2(\Omega)}+\norm{e^{-s\eta_0}q^+}^2_{L^2(\Omega)}+\norm{e^{-s\eta_0}q^-}^2_{L^2(\Omega)}+s \left( \mu^++\mu^- \right)\right), \nonumber
\end{eqnarray*}
provided $s \ge s_0$. Thus, taking $s_1 := \max(s_0,2C)$ in the above estimate, we infer from \eqref{sy3} that
\begin{eqnarray}
\label{b2}
& &\norm{e^{-s\eta_0} \left( q^+u^+_0+A\cdot\nabla u^-_0+pu^-_0 \right)}^2_{L^2(\Omega)}+\norm{e^{-s\eta_0} \left( q^-u^-_0-A\cdot\nabla u^+_0+pu^+_0 \right)}^2_{L^2(\Omega)}\\
& \leq & C s^{-\frac{3}{2}}\left( \norm{e^{-s\eta_0}A}^2_{L^2(\Omega)^n}+\norm{e^{-s\eta_0}p}^2_{L^2(\Omega)}+\norm{e^{-s\eta_0}q^+}^2_{L^2(\Omega)}+\norm{e^{-s\eta_0}q^-}^2_{L^2(\Omega)}+s \left(\mu^++\mu^-\right)\right), \nonumber
\end{eqnarray}
whenever $s \ge s_1$.
%Notice that we used that $\eta_0(x) \geq 0$ for all $x \in \Omega$ in the last line of \eqref{b2}.

The rest of the proof is to adequately choose $n+1$ initial states $u_0^k:=(u^{+,k}_0,u^{-,k}_0)^T$, $k=1,\ldots,n+1$, in order to estimate each of the four unknown functions $A$, $p$ and $q^\pm$ separately, in terms of the corresponding boundary data 
$\mu^\pm_k:=\norm{e^{-s\eta_0} \varphi^{1/2} \abs{\partial_\nu\beta}^{1/2} \partial_\nu v^{\pm,k}}^2_{L^2(\tilde{\Sigma}_*)}$, where $v^{\pm,k}$ is the solution to \eqref{sy3}
with $u_0^\pm=u_0^{\pm,k}$.

\subsection{Building $n+1$ suitable initial data} 
We proceed in two steps.\\
%%%%%%%%%%%%%%%%%%%%%%%%%%%%%%%%%%%%%%%%%%%%%%%%%%%%%%%%%%%%%%%%%%%%%%%%%%%%%%%%%%%%%%%
\noindent {\it Step 1: Estimation of $p$, $q^\pm$ and $a_n$.}
We pick $\epsilon \in (0,1)$, put
$u^{+,1}_0(x^\prime,x_n):=0$, $u_0^{-,1}(x^\prime,x_n):= \langle x_n\rangle^{-\frac{1+\epsilon}{2}}$ for all $(x^\prime,x_n)\in\Omega$ and take
$u^\pm_0=u^{\pm,1}_0$ in \eqref{b2}. For all $s \ge s_1$, we get that
\begin{eqnarray*}
& & \norm{e^{-s\eta_0} \left( 2 \langle x_n\rangle^{-\frac{1+\epsilon}{2}} p-(1+\epsilon) \langle x_n \rangle^{-\frac{5+\epsilon}{2}} x_n a_n \right)}^2_{L^2(\Omega)}+ 4 \norm{e^{-s\eta_0} \langle x_n\rangle^{-\frac{1+\epsilon}{2}}q^-}^2_{L^2(\Omega)}\\
& \leq & C s^{-\frac{3}{2}} \left( \norm{e^{-s\eta_0}A}^2_{L^2(\Omega)^n}+\norm{e^{-s\eta_0}p}^2_{L^2(\Omega)}+\norm{e^{-s\eta_0}q^+}^2_{L^2(\Omega)}+\norm{e^{-s\eta_0}q^-}^2_{L^2(\Omega)}  + s \left( \mu^{+,1}+\mu^{-,1} \right)\right),
\end{eqnarray*}
which entails that
\begin{eqnarray}
\label{c0}
& & \norm{e^{-s\eta_0}\langle x_n\rangle^{-\frac{1+\epsilon}{2}}q^-}^2_{L^2(\Omega)} \\
& \leq & C s^{-\frac{3}{2}} \left( \norm{e^{-s\eta_0}A}^2_{L^2(\Omega)^n}+\norm{e^{-s\eta_0}p}^2_{L^2(\Omega)}+\norm{e^{-s\eta_0}q^+}^2_{L^2(\Omega)}+\norm{e^{-s\eta_0}q^-}^2_{L^2(\Omega)} + s \left( \mu^{+,1}+\mu^{-,1} \right)\right) \nonumber
 \end{eqnarray}
and
\begin{eqnarray}
\label{c1}
& & \norm{e^{-s\eta_0} \left( 2 \langle x_n\rangle^{-\frac{1+\epsilon}{2}} p-(1+\epsilon) \langle x_n \rangle^{-\frac{5+\epsilon}{2}} x_n a_n \right)}_{L^2(\Omega)}^2\\
& \leq & C s^{-\frac{3}{2}} \left( \norm{e^{-s\eta_0}A}^2_{L^2(\Omega)^n}+\norm{e^{-s\eta_0}p}^2_{L^2(\Omega)}+\norm{e^{-s\eta_0}q^+}^2_{L^2(\Omega)}+\norm{e^{-s\eta_0}q^-}^2_{L^2(\Omega)} + s \left( \mu^{+,1}+\mu^{-,1} \right) \right). \nonumber
\end{eqnarray}
Doing the same with $u^\pm_0=u^{\pm,2}_0:=u^{\mp,1}_0$, we obtain for all $s \ge s_1$ that
\begin{eqnarray}
\label{c2}
& & \norm{e^{-s\eta_0}\langle x_n\rangle^{-\frac{1+\epsilon}{2}}q^+}^2_{L^2(\Omega)} \\
& \leq & C  s^{-\frac{3}{2}} \left( \norm{e^{-s\eta_0}A}^2_{L^2(\Omega)^n}+\norm{e^{-s\eta_0}p}^2_{L^2(\Omega)}+\norm{e^{-s\eta_0}q^+}^2_{L^2(\Omega)}+\norm{e^{-s\eta_0}q^-}^2_{L^2(\Omega)} + s \left( \mu^{+,2}+\mu^{-,2} \right)\right), \nonumber
\end{eqnarray}
and
\begin{eqnarray}
\label{c3}
& & \norm{e^{-s\eta_0} \left( 2 \langle x_n\rangle^{-\frac{1+\epsilon}{2}} p+(1+\epsilon) \langle x_n\rangle^{-\frac{5+\epsilon}{2}} x_n a _n  \right)}^2_{L^2(\Omega)} \\
& \leq & C s^{-\frac{3}{2}}  \left(  \norm{e^{-s\eta_0}A}^2_{L^2(\Omega)^n}+\norm{e^{-s\eta_0}p}^2_{L^2(\Omega)}+\norm{e^{-s\eta_0}q^+}^2_{L^2(\Omega)}+\norm{e^{-s\eta_0}q^-}^2_{L^2(\Omega)} + s \left( \mu^{+,2}+\mu^{-,2} \right) \right). \nonumber
\end{eqnarray}
Since $8 \norm{e^{-s\eta_0} \langle x_n \rangle^{-\frac{1+\epsilon}{2}} p}^2_{L^2(\Omega)}$ is upper-bounded by the sum of
$\norm{e^{-s\eta_0} \left( 2 \langle x_n\rangle^{-\frac{1+\epsilon}{2}} p+(1+\epsilon) \langle x_n\rangle^{-\frac{5+\epsilon}{2}} x_n a_n \right)}^2_{L^2(\Omega)}$ and
$\norm{e^{-s\eta_0} \left( 2 \langle x_n \rangle^{-\frac{1+\epsilon}{2}} p- (1+\epsilon) \langle x_n \rangle^{-\frac{5+\epsilon}{2}} x_n a_n \right)}^2_{L^2(\Omega)}$, it follows from \eqref{c1} and \eqref{c3} that
\begin{eqnarray}
\label{c4}
& & \norm{e^{-s\eta_0} \langle x_n\rangle^{-\frac{1+\epsilon}{2}}p}^2_{L^2(\Omega)} \\
%& \leq & 
%\alpha^2_0\Vert e^{-s\eta_0}\langle x_n\rangle^{-\frac{1+\epsilon}{2}}p+\frac{1+\epsilon}{2}e^{-s\eta_0}a_nx_n\langle x_n\rangle^{-\frac{5+\epsilon}{2}}\Vert^2_{L^2(\Omega)}\\
%+\alpha^2_0\Vert e^{-s\eta_0} \langle x_n \rangle^{-\frac{1+\epsilon}{2}} p-\frac{1+\epsilon}{2} e^{-s\eta_0} x_n\langle x_n \rangle^{-\frac{5+\epsilon}{2}} a_n\Vert^2_{L^2(\Omega)}\\
& \leq & 
C s^{-\frac{3}{2}} \left( \norm{e^{-s\eta_0}A}^2_{L^2(\Omega)^n}+\norm{e^{-s\eta_0}p}^2_{L^2(\Omega)}+\norm{e^{-s\eta_0}q^+}^2_{L^2(\Omega)}+\norm{e^{-s\eta_0}q^-}^2_{L^2(\Omega)} + s \sum_{i=1}^2 \left( \mu^{+,i}+\mu^{-,i} \right)  \right), \nonumber
\end{eqnarray}
whenever $s \ge s_1$.
Similarly, upon estimating $\norm{e^{-s\eta_0} \left( 2 \langle x_n\rangle^{-\frac{1+\epsilon}{2}} p+(1+\epsilon) \langle x_n\rangle^{-\frac{5+\epsilon}{2}} x_n a_n \right)}^2_{L^2(\Omega)}$ from below by the difference $\frac{(1+\epsilon)^2}{2} \norm{e^{-s\eta_0} \langle x_n\rangle^{-\frac{5+\epsilon}{2}}  x_n a_n }^2_{L^2(\Omega)}-4 \norm{e^{-s\eta_0} \langle x_n\rangle^{-\frac{1+\epsilon}{2}} p}^2_{L^2(\Omega)}$, we get  from \eqref{c3}-\eqref{c4} that
\begin{eqnarray}
\label{c5}
& & \norm{e^{-s\eta_0} \langle x_n\rangle^{-\frac{5+\epsilon}{2}} x_n a_n}^2_{L^2(\Omega)} \\
& \leq & C s^{-\frac{3}{2}} \left( \norm{e^{-s\eta_0}A}^2_{L^2(\Omega)^n}+\norm{e^{-s\eta_0}p}^2_{L^2(\Omega)}+\norm{e^{-s\eta_0}q^+}^2_{L^2(\Omega)}+\norm{e^{-s\eta_0}q^-}^2_{L^2(\Omega)} +s \sum_{i=1}^2 \left( \mu^{+,i}+\mu^{-,i} \right) \right), \nonumber
\end{eqnarray}
for all $s \ge s_1$. Bearing in mind that $\abs{x_n a_n} \ge y_* \abs{a_n}$ in $\Omega$, by virtue of the assumption \eqref{h1}, it follows from \eqref{c5} that
\begin{eqnarray}
\label{c6}
& & \norm{e^{-s\eta_0} \langle x_n\rangle^{-\frac{5+\epsilon}{2}} a_n}^2_{L^2(\Omega)} \\
& \leq & C s^{-\frac{3}{2}} \left( \norm{e^{-s\eta_0}A}^2_{L^2(\Omega)^n}+\norm{e^{-s\eta_0}p}^2_{L^2(\Omega)}+\norm{e^{-s\eta_0}q^+}^2_{L^2(\Omega)}+\norm{e^{-s\eta_0}q^-}^2_{L^2(\Omega)} +s \sum_{i=1}^2 \left( \mu^{+,i}+\mu^{-,i} \right) \right), \nonumber
\end{eqnarray}
provided we have $s \ge s_1$.\\

\noindent {\it Step 2: Estimation of the $n-1$ first components $a_j$, $j=1,\ldots,n-1$, of $A$.} 
For all $k=1,\cdots,n-1$ and all $x=(x_1,\ldots,x_n)\in \Omega$, we put $u^{\pm,k+2}_0(x):=x_k\langle x_n\rangle^{-\frac{1+\epsilon}{2}}$, substitute $u^{\pm,k+2}_0$ for $u_0^\pm$ in \eqref{sy1} and then apply Proposition \ref{pro1} to \eqref{sy3}. We get for all $s \ge s_1$ that
\begin{eqnarray*}
& & \norm{e^{-s\eta_0} \left( pu^{-,k+2}_0 +A\cdot\nabla u^{-,k+2}_0 +q^+u^{+,k+2}_0\right)}^2_{L^2(\Omega)}+\norm{e^{-s\eta_0} \left(pu^{+,k+2}_0-A\cdot\nabla u^{+,k+2}_0+q^-u^{-,k+2}_0\right)}^2_{L^2(\Omega)} \\
& \leq & C s^{-3/2} \left(\norm{e^{-s\eta_0}A}^2_{L^2(\Omega)^n}+\norm{e^{-s\eta_0}p}^2_{L^2(\Omega)}+\norm{e^{-s\eta_0}q^+}^2_{L^2(\Omega)}+\norm{e^{-s\eta_0}q^-}^2_{L^2(\Omega)}+s \left(\mu^{+,k+2}+\mu^{-,k+2}\right)\right). \nonumber
\end{eqnarray*}
Since
$\abs{pu^{\mp,k+2}_0\pm A\cdot\nabla u^{\mp,k+2}_0+q^\pm u^{\pm,k+2}_0}^2\geq\frac{\abs{A\cdot\nabla u^{\mp,k+2}_0}^2}{2}-
\abs{pu^{\mp,k+2}_0+q^\pm u^{\pm,k+2}_0}^2$, this entails that
\begin{eqnarray}
\label{d1}
& & \norm{e^{-s\eta_0}A\cdot\nabla u^{+,k+2}_0}^2_{L^2(\Omega)}+\norm{e^{-s\eta_0}A\cdot\nabla u^{-,k+2}_0}^2_{L^2(\Omega)}\\
& \leq & Cs^{-3/2} \left(\norm{e^{-s\eta_0}A}^2_{L^2(\Omega)^n}+\norm{e^{-s\eta_0}p}^2_{L^2(\Omega)}+\norm{e^{-s\eta_0}q^+}^2_{L^2(\Omega)}+\norm{e^{-s\eta_0}q^-}^2_{L^2(\Omega)}+s(\mu^{+,k+2}+\mu^{-,k+2})\right) \nonumber \\
& & +\norm{e^{-s\eta_0} \left(pu^{+,k+2}_0+q^-u^{-,k+2}_0\right)}^2_{L^2(\Omega)}+\norm{e^{-s\eta_0}\left(pu^{-,k+2}_0+q^+u^{+,k+2}_0\right)}^2_{L^2(\Omega)}.
\nonumber
\end{eqnarray}
Moreover, $\norm{e^{-s\eta_0}\left(pu^{\pm,k+2}_0+q^\mp u^{\mp,k+2}_0\right)}^2_{L^2(\Omega)}=\norm{e^{-s\eta_0}x_k\langle x_n\rangle^{-\frac{1+\epsilon}{2}}\left(p+q^\mp\right)}^2_{L^2(\Omega)}$ being upper-bounded by $2\abs{\omega}^2\left(\norm{e^{-s\eta_0}\langle x_n\rangle^{-\frac{1+\epsilon}{2}}p}^2_{L^2(\Omega)} + \norm{e^{-s\eta_0}\langle x_n\rangle^{-\frac{1+\epsilon}{2}}q^\mp}^2_{L^2(\Omega)} \right)$, \eqref{c0}, \eqref{c2}, \eqref{c4} and \eqref{d1} then yield
\begin{eqnarray*}
%\label{d2}
& & \norm{e^{-s\eta_0}A\cdot\nabla u^{+,k+2}_0}^2_{L^2(\Omega)}+\norm{e^{-s\eta_0}A\cdot\nabla u^{-,k+2}_0}^2_{L^2(\Omega)} \\
& \leq & C s^{-3/2}\left(\norm{e^{-s\eta_0}A}^2_{L^2(\Omega)^n}+\norm{e^{-s\eta_0}p}^2_{L^2(\Omega)}+\norm{e^{-s\eta_0}q^+}^2_{L^2(\Omega)}+\norm{e^{-s\eta_0}q^-}^2_{L^2(\Omega)} \right. \\
& & \left. +s\left(\sum_{i=1}^2 \left( \mu^{+,i}+\mu^{-,i} \right)+ \mu^{+,k+2}+\mu^{-,k+2} \right)\right),\ s \ge s_1, \nonumber
\end{eqnarray*}
%for $s \ge s_1$.
From this, \eqref{c5} and the estimates
$\abs{A.\nabla u^{\pm,k+2}_0}^{2} \ge \frac{1}{2} \abs{\langle x_n\rangle^{-\frac{1+\epsilon}{2}} a_{k}}^{2}-\frac{\left(1+\epsilon \right)^2}{4} \abs{\langle x_n\rangle^{-\frac{5+\epsilon}{2}} x_k x_n a_n}^{2}$ and \hfill \break
$\norm{e^{-s\eta_0}\langle x_n\rangle^{-\frac{5+\epsilon}{2}} x_k x_n a_n}_{L^2(\Omega)} \le \abs{\omega} \norm{e^{-s\eta_0}\langle x_n\rangle^{-\frac{5+\epsilon}{2}} x_n a_n}_{L^2(\Omega)}$, it then follows that
\begin{eqnarray*}
%\label{d3}
& & \norm{e^{-s\eta_0} \langle x_n\rangle^{-\frac{1+\epsilon}{2}} a_{k}}^2_{L^2(\Omega)} \\
& \le & C s^{-3/2}\left(\norm{e^{-s\eta_0}A}^2_{L^2(\Omega)^n}+\norm{e^{-s\eta_0}p}^2_{L^2(\Omega)}+\norm{e^{-s\eta_0}q^+}^2_{L^2(\Omega)}+\norm{e^{-s\eta_0}q^-}^2_{L^2(\Omega)} \right. \nonumber \\
& & \left. +s\left(\sum_{i=1}^2 \left( \mu^{+,i}+\mu^{-,i} \right)+ \mu^{+,k+2}+\mu^{-,k+2} \right)\right),\ s \ge s_1. \nonumber
\end{eqnarray*}
Summing up the above inequality over $k=1,\ldots,n-1$ and remembering \eqref{c6}, we obtain

\begin{eqnarray}
\label{d4}
& & \norm{e^{-s\eta_0} \langle x_n \rangle^{-\frac{5+\epsilon}{2}} A}^2_{L^2(\Omega)^n} \\
& \le & C s^{-3/2}\left(\norm{e^{-s\eta_0}A}^2_{L^2(\Omega)^n}+\norm{e^{-s\eta_0}p}^2_{L^2(\Omega)}+\norm{e^{-s\eta_0}q^+}^2_{L^2(\Omega)}+\norm{e^{-s\eta_0}q^-}^2_{L^2(\Omega)}  + s  \xi \right), \nonumber
\end{eqnarray}
for $s \ge s_1$, where $\xi:=\sum_{i=1}^{n+1} \left( \mu^{+,i}+\mu^{-,i} \right)$.

%%%%%%%%%%%%%%%%%%%%%%%%%%%%%%%%%%%%%%%%%%%%%%%%%%%%%%%%%%%%%%%%%%%%%%%%%%%%%
%%%%%%%%%%%%%%%%%%%%%%%%%%%%%%%%%%%%%%%%%%%%%%%%%%%%%%%%%%%%%%%%%%%%%%%%%%%%%
%%%%%%%%%%%%%%%%%%%%%%%%%%%%%%%%%%%%%%%%%%%%%%%%%%%%%%%%%%%%%%%%%%%%%%%%%%%%%

\subsection{End of the proof}
For all $y>0$ we have
\begin{eqnarray}
\label{e1}
& & \left( \langle y \rangle^{-(5+\epsilon)} - C s^{-\frac{3}{2}} \right) \left( \norm{e^{-s\eta_0}A}^2_{L^2(\Omega)^n} + \norm{e^{-s\eta_0}p}^2_{L^2(\Omega_y)}+\norm{e^{-s\eta_0}q^+}^2_{L^2(\Omega_y)}+ \norm{e^{-s\eta_0}q^-}^2_{L^2(\Omega_y)} \right)\\
& \leq & C s^{-\frac{3}{2}} \left( \norm{e^{-s\eta_0}A}^2_{L^2(\Omega \setminus \Omega)^n}+\norm{e^{-s\eta_0}p}^2_{L^2(\Omega \setminus \Omega_y)}+\norm{e^{-s\eta_0}q^+}^2_{L^2(\Omega \setminus \Omega_y)}+\norm{e^{-s\eta_0}q^-}^2_{L^2(\Omega \setminus \Omega_y)} 
+ s \xi \right), \nonumber \\
& \leq & C s^{-\frac{3}{2}} \left( \norm{A}^2_{L^2(\Omega \setminus \Omega)^n}+\norm{p}^2_{L^2(\Omega \setminus \Omega_y)}+\norm{q^+}^2_{L^2(\Omega \setminus \Omega_y)}+\norm{q^-}^2_{L^2(\Omega \setminus \Omega_y)} 
+ s \xi \right),\ s \ge s_1, \nonumber
\end{eqnarray}
by \eqref{c0}, \eqref{c2}, \eqref{c4} and \eqref{d4}, where $\Omega_y:=\omega\times(-y,y)$. Notice that in the last line of \eqref{e1}, we used that $\eta_0$ is non-negative in $\Omega$.
Moreover, for all $y \ge y_1 :=\left( (2C)^{-\frac{2}{3}}s_1 \right)^{\frac{3}{2(5+\epsilon)}}$ we have $s_y:=(2C)^{\frac{2}{3}} \langle y \rangle^{\frac{2(5+\epsilon)}{3}} \ge s_1$ and
$2 C s_y^{-\frac{3}{2}} \le \langle y \rangle^{-(5+\epsilon)}$. Therefore, applying \eqref{e1} with $s=s_y$ and using that $\eta_0(x) \le \frac{e^{2 K}}{T^2}$ for all $x \in \Omega$, we obtain that
\begin{equation}
\label{e2}
\Theta_{\Omega_y} \leq C \left( \Theta_{\Omega \setminus \Omega_y} + \langle y \rangle^{\frac{2(5+\epsilon)}{3}} \xi \right),\ y \ge y_1,
\end{equation}
where we set $\Theta_X:=\norm{A}^2_{0,X} + \norm{p}^2_{0,X}+\norm{q^+}^2_{0,X}+ \norm{q^-}^2_{0,X}$ for any subset $X \subset \Omega$.
%\begin{eqnarray*}
%\label{e1}
%& & \norm{A}^2_{L^2(\Omega)^n} + \norm{p}^2_{L^2(\Omega_y)}+\norm{q^+}^2_{L^2(\Omega_y)}+ \norm{q^-}^2_{L^2(\Omega_y)}\\
%& \leq & C e^{C \langle y \rangle^{\frac{2(5+\epsilon)}{3}}} \left( \norm{A}^2_{L^2(\Omega \setminus \Omega)^n}+\norm{p}^2_{L^2(\Omega \setminus \Omega_y)}+\norm{q^+}^2_{L^2(\Omega \setminus \Omega_y)}+\norm{q^-}^2_{L^2(\Omega \setminus \Omega_y)} + C \langle y \rangle^{\frac{2(5+\epsilon)}{3}} \xi \right),\ y \ge y_1, \nonumber
%\end{eqnarray*}
%%%%%%%%%%%%%%%%%%%%%%%%%%%%%%%%%%%%%%%%%%%%%%%%%%%%%%%%%%%%%%%%%%%%%%%%%%%%%%%
Next, using that $p_j \in \cP_{\pf}(p_0)$ for $j=1,2$, we infer from \eqref{pco} upon writing $\norm{p}_{L^2(\Omega \setminus \Omega_y)} \leq \sum_{j=1,2} \norm{p_j-p_0}_{L^2(\Omega \setminus \Omega_y)}$, that
\begin{eqnarray}
\label{e3}
\norm{p}^2_{L^2(\Omega \setminus \Omega_y)} & \leq& 4 \pf^2 \int_{\Omega \setminus \Omega_y} e^{-2\kappa \langle x_n \rangle^{\varrho}}dx^\prime dx_n\\
& \leq & 4 \pf^2\abs{\omega}  \int_{\abs{x_n}>y} e^{-2\kappa \langle x_n\rangle^{\varrho}}dx_n \nonumber\\
& \leq &4 \pf^2 \abs{\omega} \left(\int_{\R} e^{-\delta \langle x_n\rangle^{\varrho}}dx_n\right)e^{-(2\kappa-\delta)\langle y\rangle^{\varrho}},\ \delta \in (0,2\varrho). \nonumber
\end{eqnarray}
Similarly, since $q_j^\pm \in \cP_{\qf}(q_0^\pm)$ and $A_j \in \cA_{\af}(A_0)$ for $j=1,2$, we obtain
\begin{equation}
\label{e3b}
\Theta_{\Omega \setminus \Omega_y} \leq C e^{-(2\kappa-\delta)\langle y\rangle^{\varrho}},\ \delta \in (0,2\varrho),
\end{equation} 
from \eqref{aco} and \eqref{e3}, where $C=4 \abs{\omega}(\af^2+\pf^2+2 \qf^2) \int_{\R} e^{-\delta \langle x_n\rangle^{\varrho}}dx_n$. It follows from this and \eqref{e2} that
\begin{equation}
\label{e4}
\Theta_{\Omega_y} \leq C \left( e^{-(2\kappa-\delta)\langle y\rangle^{\varrho}} + \langle y \rangle^{\frac{2(5+\epsilon)}{3}} \xi \right),\ y \ge y_1,\ \delta \in (0,2\varrho).
\end{equation}

Put $\xi_1:= e^{-(2\kappa-\delta)\langle y_1 \rangle^{\varrho}}$. We shall examine the two cases $\xi \in (0,\xi_1]$ and $\xi \in (\xi_1,+\infty)$ separately. Let us start with $\xi \in (0,\xi_1]$. In this case, we pick
%since $e^{-(2\kappa-\delta)\langle y \rangle^{\varrho}} \le \xi_1$ for all $y \ge y_1$, 
$y \in [y_1,+\infty)$ so large that $e^{-(2\kappa-\delta)\langle y \rangle^{\varrho}}=\xi$, i.e., $y=\left( \left(-\frac{\ln \xi}{2 \kappa - \delta} \right)^{\frac{2}{\varrho}}-1 \right)^{\frac{1}{2}}$. Thus, with reference to \eqref{e3b}-\eqref{e4} we get for all $\xi \in (0,\xi_1]$ that
$\Theta_{\Omega \setminus \Omega_y} \le C \xi_1^{1-2\theta} \xi^{2\theta}$ and that
$\Theta_{\Omega_y} \leq C \left( \xi_1^{1-2\theta} + C_1(\theta) \right) \xi^{2\theta}$,
where $C_1(\theta):=\sup_{\xi \in (0,\xi_1]} \left( \xi^{1-2\theta} \left(\frac{-\ln \xi}{2 \kappa - \delta} \right)^{\frac{2(5+\epsilon)}{3 \varrho}} \right)<\infty$ from the assumption $\varrho>0$. As a consequence we have 
\begin{equation}
\label{e5}
\Theta_{\Omega} \leq C \left( 2\xi_1^{1-2\theta} + C_1(\theta) \right) \xi^{2\theta},\ \xi \in (0,\xi_1],
\end{equation}
and the desired result follows. Now, when $\xi \in (\xi_1,+\infty)$, we infer from \eqref{pco} upon majorizing
$\norm{p}_{L^2(\Omega)}^2$ by the sum
$2 \sum_{j=1,2} \norm{p_j-p_0}_{L^2(\Omega)}^2$, that
$\norm{p}_{L^2(\Omega)}^2\leq 4 \pf^2 \abs{\omega} \left( \int_{\R} e^{-2 \kappa \langle x_n \rangle^{\varrho}} dx_n \right) \xi_1^{-2\theta} \xi^{2\theta}$. Doing the same with $q^\pm$ and $A$, with the aid of, respectively, \eqref{pco} and \eqref{aco}, we find that $\Theta_\Omega \leq \tilde{C}_1(\theta) \xi^{2\theta}$, where the notation 
$\tilde{C}_1(\theta)$ stands for the constant $4 \left( \af^2+\pf^2 +2\qf^2 \right) \abs{\omega} \left( \int_{\R} e^{-2 \kappa \langle x_n \rangle^{\varrho}} dx_n \right) \xi_1^{-2\theta}$. This, \eqref{e5} and the estimates $\mu^\pm_k \le C \norm{\partial_\nu v^{\pm,k}}^2_{L^2(\tilde{\Sigma}_*)}$ for all $k=1,\ldots,n+1$, yield \eqref{se}, which completes the proof of Theorem \ref{thm-main}.

\section*{Acknowledgments}

\'ES is partially supported by the Agence Nationale de la Recherche (ANR) under grant ANR-17-CE40-0029.

%%%%%%%%%%%%%%%%%%%%%%%%%%%%%%%%%%%%%%

\bigskip 
%\bigskip

\noindent {\it Imen Rassas} \\
Universit\'e de Tunis El Manar, \'Ecole Nationale d'Ing\'enieurs de Tunis, LAMSIN, BP 37, Tunis Le Belv\'ed\`ere, Tunisia.\\
E-mail: imen.rassass@gmail.com. 
\bigskip

\noindent {\it Mohamed Hamrouni} \\
Universit\'e de Sousse, \'Ecole Sup\'erieure des Sciences et de la Technologie de Hammam Sousse, Rue Lamine Abassi, Hammam Sousse 4011. \\
E-mail: hamrouni.mohamed4@gmail.com.
\bigskip

\noindent {\it \'Eric Soccorsi} \\
Aix-Marseille Univ, Universit\'e de Toulon, CNRS, CPT, Marseille, France. \\
E-mail: eric.soccorsi@univ-amu.fr.

\end{document}